\documentclass{amsart}
\usepackage{amsfonts}
\usepackage{blkarray}
\usepackage{float,multirow}
\newtheorem{theorem}{Theorem}[section]

\theoremstyle{definition}
\newtheorem{definition}[theorem]{Definition}

\theoremstyle{remark}
\newtheorem{remark}[theorem]{Remark}
\numberwithin{equation}{section}

\begin{document}
\title[New extremal binary codes from $\mathbb{F}_{4}+u\mathbb{F}_{4}$-lifts
of QDC codes over $\mathbb{F}_{4}$]{New extremal binary self-dual codes from
$\mathbb{F}_{4}+u\mathbb{F}_{4}$-lifts of quadratic circulant codes over $%
\mathbb{F}_{4}$}
\author{Abidin Kaya}
\author{Bahattin Yildiz}
\author{Irfan Siap}
\address{Department of Mathematics, Fatih University, 34500, Istanbul, Turkey%
}
\email{byildiz@fatih.edu.tr, akaya@fatih.edu.tr}
\address{Department of Mathematics, Yildiz Technical University, 34210,
Istanbul, Turkey}
\email{isiap@yildiz.edu.tr}
\subjclass[2000]{Primary 94B05, 94B99; Secondary 11T71, 13M99}
\keywords{quadratic residue codes, extremal self-dual codes, Gray maps,
quadratic double-circulant codes}

\begin{abstract}
In this work, quadratic double and quadratic bordered double circulant
constructions are applied to $\mathbb{F}_4+u\mathbb{F}_4$ as well as $%
\mathbb{F}_4$, as a result of which extremal binary self-dual codes
of length 56 and 64 are obtained. The binary extension theorems as
well as the ring extension version are used to obtain 7 extremal
self-dual binary codes of length $58$, 24 extremal self-dual binary
codes of length $66$ and 29 extremal self-dual binary codes of
length $68$, all with new weight enumerators, to update the list of
all the known extremal self-dual codes in the literature.
\end{abstract}

\maketitle

\section{Introduction}

Double circulant, bordered double circulant and recently four circulant
constructions are some of the well known methods by which self-dual codes
especially the extremal ones are constructed and classified. While in the
former two, the exact conditions when the resulting codes are self-dual are
not known, in the four circulant case a necessary and sufficient condition
is given for the resulting code to be self-dual. Another common theme
related to binary self-dual codes has been using extension theorems to
obtain self-dual codes of length $n+2$ from self-dual codes of length $n$.

Recently certain binary rings (rings of characteristic 2) have been
successfully used to obtain many new extremal binary self-dual codes using
the construction methods mentioned above and extension theorems. Some of the
examples of these constructions can be found in \cite%
{karadeniz58,karadeniz66,karadeniz68,karadenizfour}, etc.

Gaborit, in \cite{gaborit} combined the double and bordered double circulant
constructions with quadratic residues modulo a prime to construct self-dual
codes of certain lengths over finite fields. The construction was called
quadratic double and quadratic bordered double circulant construction. In
this case also, the conditions for the self duality of the constructed codes
were expressed explicitly.

Gaborit's method over fields is extended to the rings by the authors in \cite%
{Kaya68} for the ring $\mathbb{F}_{2}+u\mathbb{F}_{2}+u^{2}\mathbb{F}_{2}$
with $u^{3}=u$ to obtain a considerable number of new extremal binary
self-dual codes of length $68$.

In a recent work (\cite{kayayildiz}), it was shown that the extension
theorems for binary self-dual codes can also be extended to hold for
self-dual codes over binary rings.

In this work, we combine all the above mentioned methods for some
special rings, that are the rings
$\mathbb{F}_{2^{m}}+u\mathbb{F}_{2^{m}}$ for $m=1,2 $, where
$u^{2}=0$. These rings have been studied in \cite{betsumiya} and
\cite{ling}, where it was shown that they are endowed with a
duality-preserving Gray map that allows us to construct binary
self-dual codes from self-dual codes over
$\mathbb{F}_{2^{m}}+u\mathbb{F}_{2^{m}}$. To elaborate on the main
focus of the current work, we have combined the extension theorem
for rings with quadratic double and quadratic bordered double
circulant constructions over the rings
$\mathbb{F}_{2}+u\mathbb{F}_{2} $ and
$\mathbb{F}_{4}+u\mathbb{F}_{4}$ and using these methods, we have
been able to construct a considerable number of new extremal binary
self-dual codes of lengths 58, 66 and 68. The precise numbers of new
codes obtained from the constructions that will be explained in
subsequent chapters are as follows: 7 extremal binary self-dual
codes of length $58$ with new weight enumerators in $W_{58,2}$; 24
extremal binary self-dual codes of length $66$ with new weight
enumerators in $W_{66,3}$; and 29 extremal binary self-dual codes of
length $68$ with new weight enumerators in $W_{68,2}$.

The rest of the paper is organized as follows. In Section 2, we give the
preliminaries about the rings $\mathbb{F}_{2}+u\mathbb{F}_{2}$, $\mathbb{F}%
_{4}+u\mathbb{F}_{4}$, binary self-dual codes and the extension theorems.
Section 3 includes the quadratic double and bordered double circulant
constructions for self-dual codes over $\mathbb{F}_{4}+u\mathbb{F}_{4}$ and $%
\mathbb{F}_{4}$ as well as the lifts of self-dual codes over $\mathbb{F}_{4}$%
. This leads to several extremal binary self-dual codes of lengths 56 and
64, which will form a basis for the extensions to be applied in Section 4.
Extension theorems of various forms are applied in Section 4 to the extremal
codes obtained in Section 3, as a result of which numerous new extremal
binary self-dual codes of lengths 58, 66 and 68 are obtained. We conclude
with remarks and possible directions for future research.

\section{Preliminaries}

Let $\mathbb{F}_{4}$ be the finite field of four elements, in other words $%
\mathbb{F}_{4}=\mathbb{F}_{2}\left( \omega \right) $ where $\mathbb{\omega }$
is a root the unique irreducible binary quadratic polynomial $x^{2}+x+1$.
The ring $\mathbb{F}_{4}+u\mathbb{F}_{4}$ defined via $u^{2}=0$ can be
viewed as an extension of $\mathbb{F}_{2},\mathbb{F}_{4}$ or $\mathbb{F}%
_{2}+u\mathbb{F}_{2}$. Self-dual codes over a ring of characteristic 2 are
called Type II if the Lee weights are multiples of 4 and Type I otherwise.
Type II codes over $\mathbb{F}_{4}+u\mathbb{F}_{4}$ is studied in \cite{ling}%
, the results later generalized to the ring $\mathbb{F}_{2^{m}}+u\mathbb{F}%
_{2^{m}}$ \cite{betsumiya}. For more details we refer to \cite%
{ling,betsumiya}. Recently, four circulant codes over $\mathbb{F}_{4}+u%
\mathbb{F}_{4}$ have been studied in\ \cite{kayayildiz}. An $\left( \mathbb{F%
}_{4}+u\mathbb{F}_{4}\right) $-submodule $C$ of $\left( \mathbb{F}_{4}+u%
\mathbb{F}_{4}\right) ^{n}$ is called a linear code of length $n$ over $%
\mathbb{F}_{4}+u\mathbb{F}_{4}$. Let $x=\left( x_{1},x_{2},\ldots
,x_{n}\right) $ and $y=\left( y_{1},y_{2},\ldots ,y_{n}\right) $, the dual $%
C^{\perp }$ of the code $C$ is defined for the Euclidean inner product $%
\left\langle x,y\right\rangle =\sum x_{i}y_{i}$ as
\begin{equation*}
C^{\perp }=\left\{ x\in \left( \mathbb{F}_{4}+u\mathbb{F}_{4}\right)
^{n}|\left\langle x,y\right\rangle =0\text{ for all }y\in C\right\} .
\end{equation*}

$C$ is said to be self-dual if $C=C^{\perp }$. Rains finalized the upper
bound for the minimum distance $d$ of a binary self-dual code of length $n$
in \cite{Rains} as $d\leq $ $4\lfloor \frac{n}{24}\rfloor +6$ if $n\equiv 22%
\pmod{24}$ and $d\leq $ $4\lfloor \frac{n}{24}\rfloor +4$, otherwise. A
self-dual binary code is called \textit{extremal }if it meets the bound.
Throughout the text we have used the following Gray maps from \cite{ling};
\begin{equation*}
\begin{tabular}{l|l}
$\psi :\left( \mathbb{F}_{4}+u\mathbb{F}_{4}\right) ^{n}\rightarrow \left(
\mathbb{F}_{2}+u\mathbb{F}_{2}\right) ^{2n}$ & $\phi :\left( \mathbb{F}_{2}+u%
\mathbb{F}_{2}\right) ^{n}\rightarrow \mathbb{F}_{2}^{2n}$ \\
$a\omega +b\overline{\omega }\mapsto \left( a,b\right) \text{, \ }a,b\in
\left( \mathbb{F}_{2}+u\mathbb{F}_{2}\right) ^{n}$ & $a+bu\mapsto \left(
b,a+b\right) \text{, \ }a,b\in \mathbb{F}_{2}^{n}.$%
\end{tabular}%
\end{equation*}%
The maps preserve orhogonality, therefore the Gray images of self-dual codes
are self-dual.

Quadratic double circulant (QDC) codes are introduced in \cite{gaborit}. Let
$Q_{p}\left( a,b,c\right) $ be a circulant matrix\ with first row $r$ based
on quadratic residues modulo $p$ defined as $r\left[ 1\right] =a,$ $r\left[
i+1\right] =b$ if $i$ is a quadratic residue and $r\left[ i+1\right] =c$ if $%
i$ is a quadratic non-residue, we state the special case of the main theorem
from \cite{gaborit} where $p$ is an odd prime;

\begin{theorem}
$($\cite{gaborit}\label{qdc}$)$ Let $p$ be an odd prime and let $Q_{p}\left(
a,b,c\right) $ be a quadratic residue circulant matrix with $a,b$ and $c$ as
the elements of the commutative ring $R$ with identity. If $p=4k+1$ then%
\begin{eqnarray*}
&&Q_{p}\left( a,b,c\right) Q_{p}\left( a,b,c\right) ^{T} \\
&=&Q_{p}\left( a^{2}+2k\left( b^{2}+c^{2}\right) ,2ab-b^{2}+k\left(
b+c\right) ^{2},2ac-c^{2}+k\left( b+c\right) ^{2}\right) .
\end{eqnarray*}%
If $p=4k+3$ then
\begin{eqnarray*}
&&Q_{p}\left( a,b,c\right) Q_{p}\left( a,b,c\right) ^{T} \\
&=&Q_{p}(a^{2}+\left( 2k+1\right) \left( b^{2}+c^{2}\right) ,ab+ac+k\left(
b^{2}+c^{2}\right) +\left( 2k+1\right) bc, \\
&&ab+ac+k\left( b^{2}+c^{2}\right) +\left( 2k+1\right) bc).
\end{eqnarray*}
\end{theorem}

\begin{definition}
The code generated by $P_{p}\left( a,b,c\right) =\left[
\begin{array}{c|c}
I_{p} & Q_{p}\left( a,b,c\right)%
\end{array}%
\right] $ over $R$ is called a quadratic pure double circulant code and is
denoted by $\mathcal{P}_{p}\left( a,b,c\right) $. In a similar way, the code
generated by
\begin{equation*}
B_{p}\left( a,b,c,\lambda ,\beta ,\gamma \right) =\left[
\begin{array}{c|c}
I_{p+1} &
\begin{array}{cc}
\lambda & \beta \times \boldsymbol{1} \\
\gamma \times \boldsymbol{1}^{T} & Q_{p}\left( a,b,c\right)%
\end{array}%
\end{array}%
\right] ,
\end{equation*}%
where $\boldsymbol{1}$ is the all $1$ vector of length $p$, is called a
bordered quadratic double circulant code and is denoted by $\mathcal{B}%
_{p}\left( a,b,c,\lambda ,\beta ,\gamma \right) $.
\end{definition}

In order to define an $\left( \mathbb{F}_{4}+u\mathbb{F}_{4}\right) $-lift
of a code over $\mathbb{F}_{4}$ we need the projection $\mu :\left( \mathbb{F%
}_{4}+u\mathbb{F}_{4}\right) ^{n}\mapsto \mathbb{F}_{4}^{n}$ defined as $\mu
\left( \overline{a}+\overline{b}u\right) =\overline{a}$.

\begin{theorem}
$($\cite{karadeniz68}$)$ \label{lift} $(\mathbb{F}_{2}$ replaced by $\mathbb{%
F}_{4})$ Let $C$ be a self-dual code of length $n$ over $\mathbb{F}_{4}+u%
\mathbb{F}_{4}$ then $\mu \left( C\right) $ is a self-orthogonal code of
length $n$ over $\mathbb{F}_{4}$.
\end{theorem}

\begin{definition}
Let $C$ and $D$ be codes of length $n$ over $\mathbb{F}_{4}+u\mathbb{F}_{4}$
and $\mathbb{F}_{4}$, respectively. Then, $C$ is said to be a lift of $D$ if
$\mu \left( C\right) =D$.
\end{definition}

Throughout the text we consider quadratic double circulant codes over\ $%
\mathbb{F}_{4}$, their lifts to $\mathbb{F}_{4}+u\mathbb{F}_{4}$, $\mathbb{F}%
_{2}+u\mathbb{F}_{2}$ and $\mathbb{F}_{2}$ and ring extensions of the
related self-dual codes. For extensions the following theorems from \cite%
{kayayildiz} are used. In the sequel, we let $S$ to be a commutative ring of
characteristic $2$ with identity.

\begin{theorem}
$($\cite{kayayildiz}$)$ \label{ext}Let $C$ be a self-dual code over $S$ of
length $n$ and $G=(r_{i})$ be a $k\times n$ generator matrix for $C$, where $%
r_{i}$ is the $i$-th row of $G$, $1\leq i\leq k$. Let $c$ be a unit in $S$
such that $c^{2}=1$ and $X$ be a vector in $S^{n}$ with $\left\langle
X,X\right\rangle =1$. Let $y_{i}=\left\langle r_{i},X\right\rangle $ for $%
1\leq i\leq k$. Then the following matrix%
\begin{equation*}
\left[
\begin{array}{cc|c}
1 & 0 & X \\ \hline
y_{1} & cy_{1} & r_{1} \\
\vdots & \vdots & \vdots \\
y_{k} & cy_{k} & r_{k}%
\end{array}%
\right],
\end{equation*}
generates a self-dual code $D$ over $S$ of length $n+2$.
\end{theorem}

Another extension method which can be applied to generator matrices in
standard form is as follows.:

\begin{theorem}
$($\cite{kayayildiz}$)$ \label{idext}Let $C$ be a self-dual code generated
by $G=\left[ I_{n}|A\right] $ over $S$. If the sum of the elements in $i$-th
row of $A$ is $r_{i}$ then the matrix:%
\begin{equation*}
G^{\ast }=\left[
\begin{array}{cc|cccccc}
1 & 0 & x_{1} & \ldots & x_{n} & 1 & \ldots & 1 \\ \hline
y_{1} & cy_{1} & \multicolumn{3}{c}{} & \multicolumn{3}{c}{} \\
\vdots & \vdots &  & I_{n} &  &  & A &  \\
y_{n} & cy_{n} &  &  &  &  &  &
\end{array}%
\right] ,
\end{equation*}%
where $y_{i}=x_{i}+r_{i}$, $c$ is a unit with $c^{2}=1$, $X=\left(
x_{1},\ldots ,x_{n}\right) $ and $\left\langle X,X\right\rangle =1+n$,
generates a self-dual code $C^{\ast }$ over $S$.
\end{theorem}

\section{Quadratic double circulant codes over $\mathbb{F}_{4}+u\mathbb{F}%
_{4}$}

We investigate QDC codes over $\mathbb{F}_{4}+u\mathbb{F}_{4}$ and the $%
\left( \mathbb{F}_{4}+u\mathbb{F}_{4}\right) $-lifts of QDC codes over $%
\mathbb{F}_{4}$. Note that a lift of a self-dual code may not be self-dual.
But considering the $\left( \mathbb{F}_{4}+u\mathbb{F}_{4}\right) $-lifts of
a self-dual code over $\mathbb{F}_{4}$ reduces the workload remarkably. For
instance, there are 8 free elements for a random double circulant code of
length $16$ over $\mathbb{F}_{4}+u\mathbb{F}_{4}$ which results in $%
16^{8}=4294967296$ possibilities. On the other hand, if we consider the
lifts of a unique double circulant code of length $16$ over $\mathbb{F}_{4}$
the number of possibilities reduces to $4^{8}=65536$. This is because for a
fixed element $a\in \mathbb{F}_{4}$ there are $4$ possible lifts $a+bu\in
\mathbb{F}_{4}+u\mathbb{F}_{4}$ where $b$ is an arbitrary element of $%
\mathbb{F}_{4}$. Results are obtained by computational algebra system MAGMA,
for more details we refer the reader to \cite{magma}. By using these
constructions we obtain some binary self-dual codes of lengths $56$ and $64$.

\subsection{Quadratic double circulant codes over $\mathbb{F}_{4}+u\mathbb{F}%
_{4}$}

QDC codes over $\mathbb{F}_{4}+u\mathbb{F}_{4}$ is a large family of
self-dual codes. In the following theorem we list some of these which
correspond to self-dual codes.

\begin{theorem}
Let $p$ be a prime then the following holds:\newline
if $p=8k+7$ then the codes
\begin{equation*}
\mathcal{P}_{p}\left( u,1+\omega ,\omega +u\omega \right) \text{ and }%
\mathcal{B}_{p}\left( 1+u+u\omega ,\omega +u\omega ,1+\omega ,u,1+u\omega
,1+u\omega \right)
\end{equation*}%
are Type II codes over $\mathbb{F}_{4}+u\mathbb{F}_{4}$ , \newline
if $p=8k+3$ then the codes
\begin{eqnarray*}
&&\mathcal{B}_{p}\left( 1,u+u\omega ,1,u+u\omega ,1+u+u\omega ,1+u+u\omega
\right) \text{ and } \\
&&\mathcal{B}_{p}\left( 1+\omega ,\omega +u\omega ,1+\omega ,\omega
,1+\omega +u\omega ,1+u+\omega \right)
\end{eqnarray*}%
are respectively TypeI and Type II codes over $\mathbb{F}_{4}+u\mathbb{F}%
_{4} $.
\end{theorem}

\begin{proof}
Self duality of the codes follows by Theorem \ref{qdc} and the type of the
code follows by the weights of the elements. For instance
\begin{eqnarray*}
&&Q_{p}\left( u,1+\omega ,\omega +u\omega \right) Q_{p}\left( u,1+\omega
,\omega +u\omega \right) ^{T} \\
&=&Q_{p}(u^{2}+\left( 1+\omega \right) ^{2}+\left( \omega +u\omega \right)
^{2},u\left( 1+\omega \right) +u\left( \omega +u\omega \right) +\left(
1+\omega \right) \left( \omega +u\omega \right) , \\
&&u\left( 1+\omega \right) +u\left( \omega +u\omega \right) +\left( 1+\omega
\right) \left( \omega +u\omega \right) ) \\
&=&Q_{p}\left( 1,0,0\right) =I_{p}.
\end{eqnarray*}%
Hence, the code $\mathcal{P}_{p}\left( u,1+\omega ,\omega +u\omega \right) $
is self-dual. In addition, $wt\left( 1\right) =2,\ wt\left( u\right) =4,\
wt\left( 1+\omega \right) =1$ and $wt\left( \omega +u\omega \right) =2$ so
the weight of any row of $P_{p}\left( u,1+\omega ,\omega +u\omega \right) $
is $2+4+\frac{p-1}{2}1+\frac{p-1}{2}2=6+4k+3+8k+7=4\left( 3k+4\right) $.
Thus all codewords in $\mathcal{P}_{p}\left( u,1+\omega ,\omega +u\omega
\right) $ have weight divisible by 4. The proofs of the other cases are
analogous and skipped.
\end{proof}

We finish this section by presenting some examples of quadratic double
circulant codes over $\mathbb{F}_{4}+u\mathbb{F}_{4}$ in Table \ref{tab:QDC}%
.
\begin{table}[tbp]
\caption{Quadratic double circulant codes over $\mathbb{F}_{4}+u\mathbb{F}%
_{4}$ }
\label{tab:QDC}
\begin{center}
\begin{tabular}{|c|c|}
\hline the code over $\mathbb{F}_{4}+u\mathbb{F}_{4}$ & the binary
image; $\phi \circ \psi \left( C\right) $ \\ \hline
$\mathcal{B}_{3}\left( 1+\omega ,\omega +u\omega ,1+\omega ,\omega
,1+\omega
+u\omega ,1+u+\omega \right) $ & $\left[ 32,16,8\right] _{2}$ Type II \\
\hline $\mathcal{B}_{3}\left( 1,u+u\omega ,1,u+u\omega ,1+u+u\omega
,1+u+u\omega \right) $ & $\left[ 32,16,8\right] _{2}$ Type I \\
\hline $\mathcal{P}_{7}\left( u,1+\omega ,\omega +u\omega \right) $
& $\left[ 56,28,12\right] _{2}$ Type II \\ \hline
$\mathcal{B}_{7}\left( 1+u+u\omega ,\omega +u\omega ,1+\omega
,u,1+u\omega ,1+u\omega \right) $ & $\left[ 64,32,12\right] _{2}$
Type II \\ \hline $\mathcal{B}_{11}\left( 1+\omega ,\omega +u\omega
,1+\omega ,\omega ,1+\omega +u\omega ,1+u+\omega \right) $ & $\left[
96,48,12\right] _{2}$ Type II \\ \hline $\mathcal{B}_{11}\left(
1,u+u\omega ,1,u+u\omega ,1+u+u\omega ,1+u+u\omega \right) $ &
$\left[ 96,48,12\right] _{2}$ Type I \\ \hline
\end{tabular}%
\end{center}
\end{table}

\subsection{$\left( \mathbb{F}_{4}+u\mathbb{F}_{4}\right) $-lifts of
quadratic double circulant code $\mathcal{P}_{7}\left( 0,1+\protect\omega ,%
\protect\omega \right) $ over $\mathbb{F}_{4}$}

We consider the lifts of\ the QDC code $\mathcal{P}_{7}\left( 0,1+\omega
,\omega \right) $ over $\mathbb{F}_{4}$ whose Gray image is a self-dual $%
\left[ 28,14,6\right] _{2}$ code. The code is lifted to $\mathbb{F}_{4}+u%
\mathbb{F}_{4}$ and as a result double circulant codes with binary image $%
\left[ 56,28,10\right] _{2}$ are obtained.

In the following let $\mathcal{K}_{i}$ be the $\mathbb{F}_{4}+u\mathbb{F}%
_{4} $-code generated by $\left[ I_{7}|A_{i}\right] $ where $A_{i}$ is the
circulant matrix with first row $R_{i}$.

\begin{equation*}
\begin{tabular}{|c|c|}
\hline
$\mathcal{K}_{i}$ & $R_{i}$ \\ \hline
$\mathcal{K}_{1}$ & $(u\omega ,1+\omega +u\omega ,1+u+\omega ,\omega
+u\omega ,1+u+\omega +u\omega ,\omega ,u+\omega )$ \\ \hline
$\mathcal{K}_{2}$ & $\left( u+u\omega ,1+\omega +u\omega ,1+u+\omega
+u\omega ,u+\omega ,1+\omega ,\omega ,u+\omega +u\omega \right) $ \\ \hline
$\mathcal{K}_{3}$ & $\left( u\omega ,1+\omega ,1+u+\omega ,\omega +u\omega
,1+\omega +u\omega ,u+\omega +u\omega ,\omega \right) $ \\ \hline
\end{tabular}%
\end{equation*}%
The binary images are $[56,28,10]_2$ self-dual codes, these are extended in
Section \ref{58} in order to obtain new binary codes of length $58$.

\subsection{$\left( \mathbb{F}_{4}+u\mathbb{F}_{4}\right) $-lifts of
quadratic double circulant code $\mathcal{B}_{7}\left( 1,\protect\omega ,1+%
\protect\omega ,0,1,1\right) $ over $\mathbb{F}_{4}$}

The possible weight enumerators of Type I extremal self-dual codes of length
$64$ are characterized in \cite{conway} as:
\begin{eqnarray*}
W_{64,1} &=&1+(1312+16\beta )y^{12}+(22016-64\beta )y^{14}+\cdots \text{
where }14\leq \beta \leq 104, \\
W_{64,2} &=&1+(1312+16\beta )y^{12}+(23040-64\beta )y^{14}+\cdots .\text{%
where }0\leq \beta \leq 277\text{.}
\end{eqnarray*}%
Recently, codes with $\beta =$29, 39, 53 and 60 in $W_{64,1}$ and codes with
$\beta =$51, 58 in $W_{64,2}$ are constructed in \cite{yankov64} and a code
with $\beta =80$ in $W_{64,2}$ is constructed in \cite{karadenizfour}.
Together with these the existence of such codes is now known for $\beta =$%
14, 18, 22, 25, 32, 36, 39, 44, 46, 53, 60, 64 in $W_{64,1}$ and for $\beta =
$0, 1, 2, 4, 5, 6, 8, 9, 10, 12, 13, 14, 16, 17, 18, 20, 21, 22, 23, 24, 25,
28, 29, 30, 32, 33, 36, 37, 38, 40, 41, 44, 48, 51, 52, 56, 58, 64, 72, 80,
88, 96, 104, 108, 112, 114, 118, 120, 184 in $W_{64,2}$.

The code $\mathcal{B}_{7}\left( 1,\omega ,1+\omega ,0,1,1\right) $ over $%
\mathbb{F}_{4}$ has a binary image as the $\left[ 32,16,8\right] _{2}$
self-dual code with an automorphism group of order $2^{15}\times 3^{2}\times
5\times 7$. The code is lifted to $\mathbb{F}_{4}+u\mathbb{F}_{4}$ and as
binary image of the lifts extremal Type I codes of length $64$ listed in
Table \ref{tab:64codes} are obtained. In Table \ref{tab:64codes} $\mathcal{L}%
_{i}$ is the code over $\mathbb{F}_{4}+u\mathbb{F}_{4}$ generated by $G_{i}=%
\left[
\begin{array}{c|c}
I_{8} &
\begin{array}{cc}
a & b\times \boldsymbol{1} \\
c\times \boldsymbol{1}^{T} & A_{i}%
\end{array}%
\end{array}%
\right] $ where $A_{i}$ is the circulant matrix with first row $R_{i}$ and $%
\boldsymbol{1}$ is the all $1$ row vector of length $7$ and the binary code $%
\phi \circ \psi \left( \mathcal{L}_{i}\right) $ has weight enumerator with $%
\beta _{i}$ in $W_{64,1}$.
\begin{table}[tbp]
\caption{Codes in $W_{64,1}$ as lifts of $B_{7}\left( 1,\protect\omega ,1+%
\protect\omega ,0,1,1\right) $ over $\mathbb{F}_{4}$ }
\label{tab:64codes}
\begin{center}
\begin{tabular}{|c|c|c|}
\hline
$\mathcal{L}$ & $R_{i}$; first row of $A_{i}$ &  \\ \cline{2-2}
& {$\left( a,b,c\right) $; the border} & $\beta $ \\ \hline\hline
$\mathcal{L}_{1}$ & $\left( 1,u+\omega +u\omega ,\omega ,1+\omega +u\omega
,\omega +u\omega ,1+\omega ,1+u+\omega \right) $ &  \\ \cline{2-2}
& {$\left( u\omega ,1+u+u\omega ,1+u\omega \right) $} & $18$ \\ \hline\hline
$\mathcal{L}_{2}$ & $\left( 1+u\omega ,u+\omega ,\omega ,1+u+\omega +u\omega
,u+\omega +u\omega ,1+u+\omega ,1+\omega \right) $ &  \\ \cline{2-2}
& {$\left( u\omega ,1+u+u\omega ,1+u\omega \right) $} & $32$ \\ \hline\hline
$\mathcal{L}_{3}$ & $\left( 1+u,\omega ,\omega ,1+\omega ,\omega ,1+u+\omega
+u\omega ,1+u+\omega \right) $ &  \\ \cline{2-2}
& {$\left( u+u\omega ,1+u+u\omega ,1+u\omega \right) $ } & $46$ \\
\hline\hline
$\mathcal{L}_{4}$ & $\left( 1+u+u\omega ,\omega +u\omega ,\omega ,1+\omega
+u\omega ,u+\omega ,1+\omega ,1+u+\omega \right) $ &  \\ \cline{2-2}
& {$\left( u+u\omega ,1+u\omega ,1+u\omega \right) $} & $60$ \\ \hline
\end{tabular}%
\end{center}
\end{table}

\begin{remark}
The code with $\beta =60$ in $W_{64,1}$ was recently constructed in \cite%
{yankov64} and all the codes in Table \ref{tab:64codes} have an automorphism
group of order $2^{2}7$.
\end{remark}

\section{New extremal binary self-dual codes of lengths 58, 66 and 68}

In this section, $\mathbb{F}_{2}$ and $\left( \mathbb{F}_{2}+u\mathbb{F}%
_{2}\right) $-extensions of the codes constructed in the previous
section are investigated. We were able to obtain 7 new binary codes
of length 58, 24 new binary codes of length 66 and 29 new binary
codes of length 68. All these codes have new weight enumerators.

\subsection{New extremal binary self-dual codes of lengths 58\label{58}}

The possible weight enumerators of an extremal self-dual binary code of
length $58$ are characterized in \cite{conway}\ as follows:%
\begin{eqnarray*}
W_{58,1} &=&1+(165-2\beta )y^{10}+(5078+2\beta )y^{12}+\cdots \text{ where }%
0\leq \beta \leq 82 \\
W_{58,2} &=&1+(319-24\beta -2\gamma )y^{10}+\cdots \text{ where }0\leq \beta
\leq 11\text{ and }0\leq \gamma \leq 159-2\beta .
\end{eqnarray*}%
Recently, Yankov and Lee have presented the known extremal binary self-dual
codes of length $58$ and they obtain new ones in \cite{yankov}. By
considering lifts of $\left[ 8,4,4\right] _{2}$ Hamming code Karadeniz and
Kaya constructed 11 new codes in $W_{58,2}$ in \cite{karadeniz58}. Together
with the ones added from \cite{yankov} and \cite{karadeniz58}, the existence
of such codes is known for $\beta =55$ in $W_{58,1}$ and for $\beta =0$ with
$\gamma \in \{2m|m=$0, 1, 8, 9, 10, 15, 16, 34, 71, 79 or $18\leq m\leq 64\},
$ $\beta =1$ with $\gamma \in \left\{ 2m|21\leq m\leq 57\right\} $ and $%
\beta =2$ with $\gamma \in \{2m|m=$16, 18, 20, 21, 22, 19, 46, 49 or $24\leq
m\leq 44\}$ in $W_{58,2}.$

In this section, we obtain the codes in $W_{58,2}$ for $\beta =0$ with $%
\gamma =28$, $\beta =1$ with $\gamma =28,32,34,38,40$ and for $\beta =2$
with $\gamma =46$.

The binary Gray images of the codes $\mathcal{K}_{1},\mathcal{K}_{2}$ and $%
\mathcal{K}_{3}$ are self-dual $\left[ 56,28,10\right] _{2}$ codes and when
they are extended by Theorem \ref{ext} seven extremal self-dual codes of
length $58$ with new weight enumerators are obtained which are listed in
Table \ref{tab:58codes}.
\begin{table}[]
\caption{ New codes in $W_{58,2}$ by Theorem \protect\ref{ext} (7 codes)}
\label{tab:58codes}
\begin{center}
\begin{tabular}{|c|c|c|c|}
\hline
$\mathcal{K}$ & $X=(x_1,x_2,...,x_{56})$ & $\beta $ & $\gamma $ \\ \hline
$\mathcal{K}_{1}$ & $%
10011100000111010110110110010100100100010011110101110000 $ & $1$ & $32$ \\
\hline
$\mathcal{K}_{1}$ & $%
10010010011100001101111010001111110011001111000101001111 $ & $1$ & $38$ \\
\hline
$\mathcal{K}_{1}$ & $%
00111010100110000101110000110100011111011000001011001000 $ & $1$ & $40$ \\
\hline
$\mathcal{K}_{1}$ & $%
00111111011100001001000010010000010000010000101010001000 $ & $2$ & $46$ \\
\hline
$\mathcal{K}_{2}$ & $%
11010000110101000111011001001100010010100011100111101010 $ & $0$ & $28$ \\
\hline
$\mathcal{K}_{2}$ & $%
01101011110110000100001101010100111100100100110110010111 $ & $1$ & $34$ \\
\hline
$\mathcal{K}_{3}$ & $%
01110000010011000101111110000111010100000010101001100110 $ & $1$ & $28$ \\
\hline
\end{tabular}%
\end{center}
\end{table}

\subsection{New extremal binary self dual codes of length $66$}

A self-dual $\left[ 66,33,12\right] _{2}$-code has a weight enumerator in
one of the following forms (\cite{dougherty1})%
\begin{eqnarray*}
W_{66,1} &=&1+\left( 858+8\beta \right) y^{12}+\left( 18678-24\beta \right)
y^{14}+\cdots \text{ where }0\leq \beta \leq 778, \\
W_{66,2} &=&1+1690y^{12}+7990y^{14}+\cdots \text{ } \\
\text{and }W_{66,3} &=&1+\left( 858+8\beta \right) y^{12}+\left(
18166-24\beta \right) y^{14}+\cdots \text{ where }14\leq \beta \leq 756,
\end{eqnarray*}%
Recently, five new codes in $W_{66,1}$ are constructed in \cite%
{karadenizfour}. For a list of known codes in $W_{66,1}$ we refer to \cite%
{karadenizfour,tsai} and references therein.

First codes with a weight enumerator in $W_{66,3}$ were discovered in \cite%
{tsai} and recently 14 codes were discovered in \cite{karadeniz66} together
with these the existence of codes in $W_{66,3}$ is known for $\beta =$28,
29, 30, 31, 32, 33, 34, 49, 50, 54, 55, 56, 57, 58, 59, 62, 63 and 66.

In this work, we construct 24 codes with new weight enumerators, more
precisely the codes with $\beta =$35, 36, 37, 38, 43, 44, 47, 48, 51, 60,
67, 70, 71, 73, 74, 75, 76, 77, 78, 79 and 80 in $W_{66,3}$.
\begin{table}[tbp]
\caption{New codes in $W_{66,3}$ by Theorem \protect\ref{idext} (13 codes)}
\label{tab:66identity}
\begin{center}
\begin{tabular}{|c|c|c|}
\hline
$C$ & $X=(x_1,x_2,...,x_{32})$ & $\beta $ \\ \hline
$\mathcal{L}_{1}$ & $01101010101010101101101110100100$ & $35$ \\ \hline
$\mathcal{L}_{1}$ & $00100101110100100000100110110001$ & $36$ \\ \hline
$\mathcal{L}_{1}$ & $10110000001001110101100111100001$ & $37$ \\ \hline
$\mathcal{L}_{1}$ & $11100001011100101111010110010001$ & $38$ \\ \hline
$\mathcal{L}_{2}$ & $01011001111111001001101010111001$ & $43$ \\ \hline
$\mathcal{L}_{2}$ & $10110110001011100111101100111001$ & $44$ \\ \hline
$\mathcal{L}_{2}$ & $00000010110011001100010111000011$ & $47$ \\ \hline
$\mathcal{L}_{2}$ & $01101001101001111110011111100100$ & $48$ \\ \hline
$\mathcal{L}_{2}$ & $10010010110101101101100111011110$ & $51$ \\ \hline
$\mathcal{L}_{3}$ & $11001001111101100011101001000101$ & $60$ \\ \hline
$\mathcal{L}_{3}$ & $01101001010111101100000011110011$ & $67$ \\ \hline
$\mathcal{L}_{4}$ & $01010111010000001100100111001110$ & $70$ \\ \hline
$\mathcal{L}_{4}$ & $00101101101000101100010011101100$ & $75$ \\ \hline
\end{tabular}%
\end{center}
\end{table}

\begin{remark}
In order to apply the extension in Theorem \ref{idext} the generator
matrices for the binary image of $\mathcal{L}_{i}$ are converted into
standard form $\left[ I_{32}|A_{i}\right] $ and the matrices are available
online at \cite{kaya}. So, in Table \ref{tab:66identity} the codes are
generated by the matrices of the form
\begin{equation*}
\left[
\begin{array}{cc|cccccc}
1 & 0 & x_{1} & \ldots & x_{32} & 1 & \ldots & 1 \\ \hline
y_{1} & y_{1} & \multicolumn{3}{c}{} & \multicolumn{3}{c}{} \\
\vdots & \vdots &  & I_{32} &  &  & A_{i} &  \\
y_{32} & y_{32} &  &  &  &  &  &
\end{array}%
\right] .
\end{equation*}
\end{remark}

In addition, as binary extensions by Theorem \ref{ext} we obtained 11 new
codes in $W_{66,3}$ which are listed in Table \ref{tab:66codes}.
\begin{table}[tbp]
\caption{New codes in $W_{66,3}$ by Theorem \protect\ref{ext} (11 codes)}
\label{tab:66codes}
\begin{center}
\begin{tabular}{|c|c|c|}
\hline
$C$ & $X=(x_1,x_2,...,x_{64})$ & $\beta $ \\ \hline
$\mathcal{L}_{2}$ & $%
1100101011011100011110010011111101101001011100101110010100111111$ & $45$ \\
\hline
$\mathcal{L}_{2}$ & $%
0011110011001011010001000001101110011100101111000010000111011100$ & $46$ \\
\hline
$\mathcal{L}_{3}$ & $%
0000100000001000110000011100010000001011101000001000001011110011$ & $61$ \\
\hline
$\mathcal{L}_{4}$ & $%
0001101001010010010010111101111101111110010111100001000011010001$ & $71$ \\
\hline
$\mathcal{L}_{4}$ & $%
1111101001110110001000111011111101101101101011011000011111101110$ & $73$ \\
\hline
$\mathcal{L}_{4}$ & $%
1110011001010110000001010111001000001110011111010101110110010101$ & $74$ \\
\hline
$\mathcal{L}_{4}$ & $%
1010110101111010100111100111010001100000011100101110000000011001$ & $76$ \\
\hline
$\mathcal{L}_{4}$ & $%
0111111111011001001011011001001111010101011001011110110101000010$ & $77$ \\
\hline
$\mathcal{L}_{4}$ & $%
1100101100011101111011110101010011000001001010101111000111101111$ & $78$ \\
\hline
$\mathcal{L}_{4}$ & $%
1011110101001100001100101110011000001110100011010110001011110011$ & $79$ \\
\hline
$\mathcal{L}_{4}$ & $%
0000100100100100111100111011011010100100000001000101100111110110$ & $80$ \\
\hline
\end{tabular}%
\end{center}
\end{table}

\begin{remark}
In Tables \ref{tab:58codes} and \ref{tab:66codes} the codes are generated by
the matrices of the form
\begin{equation*}
\left[
\begin{array}{cc|c}
1 & 0 & X \\ \hline
y_{1} & y_{1} &  \\
\vdots & \vdots & \phi \circ \psi \left( C\right) \\
y_{k} & y_{k} &
\end{array}%
\right] \text{.}
\end{equation*}
\end{remark}

\subsection{New extremal binary self dual codes of length $68$}

The weight enumerator of an extremal binary self-dual code code of length 68
is in one of the following forms (\cite{dougherty1}):
\begin{eqnarray*}
W_{68,1} &=&1+\left( 442+4\beta \right) y^{12}+\left( 10864-8\beta \right)
y^{14}+\cdots \text{ where }104\leq \beta \leq 1358\text{,} \\
W_{68,2} &=&1+\left( 442+4\beta \right) y^{12}+\left( 14960-8\beta
-256\gamma \right) y^{14}+\cdots \text{ }
\end{eqnarray*}%
where $0\leq \gamma \leq 11$ and $14\gamma \leq \beta \leq 1870-32\gamma $.
Tsai et al. constructed a substantial number of codes in both possible
weight enumerators in \cite{tsai}. Recently, 178 new codes are obtained in
\cite{kayayildiz} and $28$ new codes including the first examples with $%
\gamma =4$ and $\gamma =6$ in $W_{68,2}$ are obtained in \cite{karadeniz68}.
Together with the ones in \cite{karadeniz68, kayayildiz} codes exists for $%
W_{68,2}$ when $\gamma =0$ and $\beta =$38, 40, 44,...,139, 141, 142, 143,
145, 147, 148, 149, 151, 153, 170, 204, 238, 272; $\gamma =1$ and $\beta =$%
54, 56, 58, 60, ..., 66, 68, 70, 72,\ldots , 115, 118, 119, 120, 123, 125,
126, 129, 132, 133, 135, 137, ...,149, 150, 151, 153, 155, 159; $%
\gamma =2$ and $\beta =$65, 68, 71, 77, 85, 87, 89, 91, 93, 95, 97,
99, 101, 103, 105, 109, 111, 115, 117, 119, 121, 123, 125, 127, 129,
131, 133, 135, 137, 139, 145, 151, 153, 155, 158, 160, 162 or $\beta
\in \left\{ 2m|37\leq m\leq 68,\text{ }70\leq m\leq 76\right\} $;
$\gamma =3$ and $\beta =$88, 90, 96, 100, 102, 104, 108, 112, 114,
116, 117, 126, 127, 128, 130, 133, 136, 137, 138, 140, 141, 142,
144, 145, 147, 148, 149, 153, 154, 158, 159, 160, 162, 176, 188,
193, 196; $\gamma =4$ and $\beta =$102, 110, 116, 120, 122, 124,
128, 130, 134, 136, 138, 140, 142, 150, 152, 154, 156, 158, 160,
162, 164, 166, 168, 170, 172, 174, 176, 180 and $\gamma =6$ with
$\beta =$138, 154, 156, 158, 162, 176, . For a list of known codes
in $W_{68,1}$ we refer to \cite{tsai}.

In this work as binary images of $\left( \mathbb{F}_{2}+u\mathbb{F}%
_{2}\right) $-extensions of codes in Table \ref{tab:68codes} we
obtain 29
codes with new weight enumerators in $W_{68,2}$, more precisely codes with $%
\gamma =1$, $\beta =$67, 69, 71, 116, 117, 121, 122, 124, 127, 128,
130, 131, 134, 136, 157; with $\gamma =2$, $\beta =$107, 113, 143,
147, 149, 154, 156, 159 and codes with rare parameters $\gamma =3$,
$\beta =$101, 110, 122, 123, 132, 156.
\begin{table}[tbp]
\caption{New codes in $W_{68,2}$ by Theorem \protect\ref{ext} for $\mathbb{F}%
_{2}+u\mathbb{F}_{2}$ (29 codes) } \label{tab:68codes}
\begin{center}
\begin{tabular}{|c|c|c|c|c|}
\hline
$C$ & $X=(x_1,x_2,...,x_{32})$ & $c$ & $\gamma $ & $\beta $ \\ \hline
$\mathcal{L}_{1}$ & $[0300u01333u330uuu10011100131u111]$ & $1+u$ & $1$ & $67$
\\ \hline
$\mathcal{L}_{1}$ & $[33010u0u003u11u013331uu13u130031]$ & $1$ & $1$ & $69$
\\ \hline
$\mathcal{L}_{1}$ & $[u311uu0uu0101u01u0000u11310uu1u0]$ & $1$ & $1$ & $71$
\\ \hline
$\mathcal{L}_{1}$ & $[3u0uu1u33030110031u1330u0011uuu1]$ & $1+u$ & $2$ & $%
107 $ \\ \hline
$\mathcal{L}_{1}$ & $[013311u1uuu1u103uu111300u3100u31]$ & $1$ & $3$ & $101$
\\ \hline
$\mathcal{L}_{1}$ & $[00311u000u0uu101u301u13030111000]$ & $1$ & $3$ & $110$
\\ \hline
$\mathcal{L}_{1}$ & $[0 u 1 1 1 u 0 0 0 u 0 0 0 3 u 1 0 3 u 3 u 1 1
0 3 0 1 3 1 u0 u]$ & $1$ & $3$ & $122$
\\ \hline
$\mathcal{L}_{1}$ & $[111u010u0u003u113001010103111031]$ & $1$ & $3$
& $123$
\\ \hline
$\mathcal{L}_{2}$ & $[u11u0303u3u0113330uuu3031110u130]$ & $1+u$ & $1$ & $%
117 $ \\ \hline
$\mathcal{L}_{2}$ & $[30033u31u30301133uu13u311311101u]$ & $1$ & $2$ & $113$
\\ \hline
$\mathcal{L}_{2}$ & $[11101u3001100uu3u133u1u3001u1130]$ & $1+u$ & $3$ & $%
132 $ \\ \hline
$\mathcal{L}_{3}$ & $[uu3113000031u11310011333u3u13u10]$ & $1$ & $1$ & $116$
\\ \hline
$\mathcal{L}_{3}$ & $[30113uuuuu33uu03u1u3uu311u00u0u1]$ & $1$ & $1$ & $124$
\\ \hline
$\mathcal{L}_{3}$ & $[313310uu3010u3003u003030u3031003]$ & $1+u$ & $1$ & $%
134 $ \\ \hline
$\mathcal{L}_{4}$ &
$[1003u10u1030u10101u1001u00uuuu1u]$ & $1$ & $1$ & $121$
\\ \hline
$\mathcal{L}_{4}$ & $[1000133u3010013uu13uuu11u031uuu1]$ & $1$ & $1$ & $122$
\\ \hline
$\mathcal{L}_{4}$ & $[130u31133u31uu30u100u03u3110u31u]$ & $1+u$ & $1$ & $%
127 $ \\ \hline
$\mathcal{L}_{4}$ & $[101103310u133003000033u3uu111uu1]$ & $1+u$ & $1$ & $%
128 $ \\ \hline
$\mathcal{L}_{4}$ & $[30u103313u031100u13uu101u03011u3]$ & $1$ & $1$ & $130$
\\ \hline
$\mathcal{L}_{4}$ & $[1u1u10uu0uu0u31u0u103u133u11333u]$ & $1+u$ & $1$ & $%
131 $ \\ \hline
$\mathcal{L}_{4}$ & $[10u011u001u310300330u03uu3uuu101]$ & $1+u$ & $1$ & $%
136 $ \\ \hline
 $\mathcal{L}_{4}$ &
$[330u0113u0uu3333u003u000uu010u33]$ & $1$ & $1$ & $157$
\\ \hline
$\mathcal{L}_{4}$ & $[011uu103103111u00030031u0110331u]$ & $1$ & $2$ & $143$
\\ \hline
$\mathcal{L}_{4}$ & $[00u1u0u1u1u30130u31u030011u00130]$ & $1+u$ & $2$ & $%
147 $ \\ \hline
$\mathcal{L}_{4}$ & $[u u u 1 0 u u 3 0 1 0 3 u 3 1 0 u 3 1 0 u 3 u u 1 1 0 0 u 3 1 u]$ & $1+u$ & $2$ & $%
149 $ \\ \hline
$\mathcal{L}_{4}$ & $[33uuu0u3uu110330u3010u110301u101]$ & $1+u$ & $2$ & $%
154 $ \\ \hline
$\mathcal{L}_{4}$ & $[330uu003 0 0 3 3 0 3 1 0 0 1 0 1 0 u 3 1 0 1 0 3 u1 u 1]$ & $1$ & $2$ & $%
156 $ \\ \hline
$\mathcal{L}_{4}$ & $[0 0 0 1 u u u 3 u 3 u 3 u 3 1 0 0 3 3 0 u 3 0 0 1 10 0 0 3 3 u]$ & $1$ & $2$ & $%
159 $ \\ \hline
$\mathcal{L}_{4}$ & $[3031u111133u01u133u1u00113303301]$ & $1+u$ & $3$ & $%
156 $ \\ \hline
\end{tabular}%
\end{center}
\end{table}

\begin{remark}
In Table \ref{tab:68codes} for the extension vectors $X$ over $\mathbb{F}%
_{2}+u\mathbb{F}_{2}$ the element $1+u$ is abbreviated as $3$ and the codes
are generated by the matrices of the following form
\begin{equation*}
G=\left[
\begin{array}{cc|c}
1 & 0 & X \\ \hline
y_{1} & cy_{1} &  \\
\vdots  & \vdots  & \psi \left( \mathcal{L}_{i}\right)  \\
y_{k} & cy_{k} &
\end{array}%
\right] \text{.}
\end{equation*}%
The binary code generated by $\phi \left( G\right) $ has a new weight
enumerator in $W_{68,2}$ for the given parameters. The binary generator
matrices for codes with $\gamma =3$ in $W_{68,2}$ in Table \ref{tab:68codes}
are available online at \cite{kaya}.
\end{remark}

\end{document}